\documentclass[11pt]{article}
\usepackage{amssymb}
\usepackage{graphicx}
\usepackage{setspace}
  \usepackage{paralist}
  \usepackage{longtable}
   \usepackage{multirow}
    \usepackage{rotating}
\usepackage{marginnote}
\usepackage{mathrsfs}
\usepackage{amsmath, amsthm, amssymb}
\usepackage{authblk}
\usepackage{graphicx}
\usepackage{float}
\usepackage{hyperref}
\usepackage[margin=1in]{geometry}
\usepackage{comment}
\usepackage{color}
\usepackage{cite}
\usepackage{indentfirst}
\allowdisplaybreaks[4]
\numberwithin{equation}{section}
\date{}
\newtheorem{theorem}{Theorem}[section]
\newtheorem*{theorem*}{Theorem}

\newtheorem{lemma}[theorem]{Lemma}
\newtheorem{corollary}[theorem]{Corollary}

\theoremstyle{remark}
\newtheorem{remark}{Remark}

\usepackage{authblk}%package for author name and affiliation

\begin{document}
\renewcommand{\thefootnote}{\fnsymbol{footnote}}

% \footnotetext{1. School of Mathematical Sciences, CMA-Shanghai, Shanghai Jiao Tong University, China( xmq157@sjtu.edu.cn);}
% \footnotetext{The first author is partially supported by}
\footnotetext{The first and second authors are partially supported by NSFC-12031012, NSFC-11831003 and the Institute of Modern Analysis-A Frontier Research Center of Shanghai.}
\footnotetext{The third author is partially supported by Natural Science Foundation of He’nan Province of China (Grant No. 222300420499).}
\renewcommand{\thefootnote}{\arabic{footnote}}

\title{A priori estimates for higher-order fractional Laplace equations}

% \author{  }

\author[1]{Yugao Ouyang}
\author[2]{Meiqing Xu}
\author[3]{Ran Zhuo}
% \affil[1]{School of Mathematics and Statistics, Xi'an Jiaotong University}
\affil[1,2]{School of Mathematical Sciences, Shanghai Jiao Tong University, Shanghai, China}
% \affil[2]{School of Mathematical Sciences, Shanghai Jiao Tong University}
\affil[3]{Mathematics and Science College,
Shanghai Normal University
Shanghai, China; Department of Mathematics and Statistics, Huanghuai University, Henan, China}

\maketitle
\begin{abstract}
\noindent
% In this paper, we make a priori estimates on high-order fractional Laplacian equation. To overcome the difficulty due to integer part of Laplacian, we split the equation into a system and make estimates on every equation in the system. The core of our strategy is to find a proper parameter to scale the domain. Such a parameter can bring a uniform bound for all equations in the system, which allows us to apply Schauder estimates on the system.

% In this paper, we use the method of moving planes to give the Liouville theorem of the higher-order fractional Laplacian equations under the local boundedness condition. Then we establish a priori estimates for the equations. To address the challenge posed by the integer part of the Laplacian operator, we divide the higher-order fractional Laplacian equation into a system, and provide estimates for each equation in the system. The crux of our approach is to identify a suitable scaling parameter for the domain, which yields a uniform bound for all equations in the system. 

In this paper, we establish a priori estimates for the positive solutions to a higher-order fractional Laplace equation on a bounded domain by a blowing-up and rescaling argument. To overcome the technical difficulty due to the high-order and fractional order mixed operators, we divide the high-order fractional Laplacian equation into a system, and provide uniform estimates for each equation in the system. Finding a proper scaling parameter for the domain is the crux of rescaling argument to the above system, and the new idea is introduced in the rescaling proof, which may hopefully be applied to many other system problems. In order to derive a contradiction in the blowing-up proof, combining the moving planes method and suitable Kelvin transform, we prove a key Liouville-type theorem under a weaker regularity assumption in a half space. 

% This, in turn, enables us to apply Schauder estimates to the system, and establish a priori bounds for the higher-order fractional Laplacian equation.
% \begin{equation}
% \left\{
%     \begin{aligned}
%   &(-\Delta)^{\frac{\alpha}{2}+m}u(x)=u^p(x),\quad x\in \Omega,\\
%   &(-\Delta)^i u= 0,\quad x\in \partial\Omega,\quad i=1,2,...,m-1,\\
%   &(-\Delta)^m u \equiv 0,\quad x\in \Omega^c.
% \end{aligned}
% \right.
% \end{equation}
% We use a direct blowing-up and rescaling argument to derive a priori estimates on positive solutions. This method was first developed to obtain a priori estimates of fractional Laplacian equation, and we extend the result to high-order fractional Laplacian.

\noindent{\bf{Keywords}}: Higher-order fractional Laplacian, a priori estimates, Liouville type theorem, blowing-up, rescaling.

\noindent{\bf {MSC 2010}}: 35B09, 35B44, 35B45, 35R11  % pseudodifferential operators.
\end{abstract}

%\tableofcontents
\section{Introduction}

The derivation of suitable a priori estimates plays an important role in the modern theory of nonlinear partial differential equations. Such estimates are used to derive the existence of solutions and the regularity estimates that ensure that generalized
solutions are actually smooth in some sense.
In this paper, we consider the
question of obtaining a priori estimates for the solutions of the following higher order fractional Laplace equation
\begin{equation} \label{in-1}
\left\{
    \begin{aligned}
  &(-\Delta)^{\frac{\alpha}{2}+m}u(x)=u^p(x),\quad x\in \Omega,\\
  &(-\Delta)^i u= 0,\quad x\in \partial\Omega,\quad i=0,1,...,m-1,\\
  &(-\Delta)^m u \equiv 0,\quad x\in \Omega^c.
\end{aligned}
\right.
\end{equation}
where $\Omega$ is a bounded domain with $C^2$ boundary, and $n\geq 2$, $0<\alpha<2$, $m\in\mathbb{N}^+$.

For $\alpha$ taking any real number between 0 and 2, $(-\Delta)^\frac{\alpha}{2}$ is a nonlocal differential operator defined by
\begin{equation}\label{fraction definition}
  (-\Delta)^\frac{\alpha}{2}u(x)=C_{n,\alpha}P.V. \int_{\mathbb{R}^n} \frac{u(x)-u(y)}{|x-y|^{n+\alpha}}dy.
\end{equation}
where P.V. is the Cauchy principal value. The integral on the right-hand
side of (\ref{fraction definition}) is well defined for $u\in C_{loc}^{[\alpha],\{\alpha\}+\epsilon}(\mathbb{R}^n) \cap\mathcal{L}_\alpha(\mathbb{R}^n)$, where $\epsilon>0$ is arbitrary small, $[\alpha]$ denotes the integer part of $\alpha$, $\{\alpha\}=\alpha-[\alpha]$,
\begin{equation}
  \mathcal{L}_\alpha(\mathbb{R}^n)=\{u:\mathbb{R}^n\rightarrow\mathbb{R}| \int_{\mathbb{R}^n} \frac{|u(x)|}{(1+|x|^{n+\alpha})}dx <\infty\}.
\end{equation}
% The condition $u\in\mathcal{L}_\alpha(\mathbb{R}^n)$ is to ensure the integrability on outside domain of the integral in (\ref{fraction definition}). The condition $u\in C_{loc}^{[\alpha],\{\alpha\}+\epsilon}(\mathbb{R}^n)$ is to ensure the integrability around $y=x$ and the continuity of $(-\Delta)^\frac{\alpha}{2}u(x)$. For example, when $0<\alpha<1$, we have
% \begin{align}
%   \int_{B_1(x)}\frac{|u(x)-u(y)|}{|x-y|^{n+\alpha}}dy \leq \int_{B_1(x)}\frac{C}{|x-y|^{n-\epsilon}}dy<\infty.
% \end{align}
% When $1\leq \alpha <2$,
% \begin{align}
%     &PV\int_{B_1(x)}\frac{u(x)-u(y)}{|x-y|^{n+\alpha}}dy\\
%     = &\lim_{\epsilon\to 0}\int_{B_1(x)\backslash B_\epsilon(x)}\frac{-\nabla u(x)\cdot(y-x)+ O(|y-x|^2)}{|y-x|^{n+\alpha}} dy\\
%     = & \lim_{\epsilon\to 0}\int_{B_1(x)\backslash B_\epsilon(x)}\frac{O(|y-x|^2)}{|y-x|^{n+\alpha}} dy\\
%     \leq &\lim_{\epsilon\to 0}\int_{B_1(x)\backslash B_\epsilon(x)}\frac{1}{|y-x|^{n+\alpha-2}}dy<\infty.
% \end{align}
For more details see \cite{silvestre,chen2020fractional}.

% In this paper, the high-order fractional operator $(-\Delta)^{\frac{\alpha}{2}+m}$ is defined by $(-\Delta)^\frac{\alpha}{2}\circ(-\Delta)^m$, i.e. we first apply the classical Laplacian operator and then apply the fractional Laplacian operator. So it is natural to require $(-\Delta)^m u\in\mathcal{L}_\alpha(\mathbb{R}^n)$.
% Also one may notice that $(-\Delta)^\frac{\alpha}{2}$ is a non-local operator. That means, the value of $(-\Delta)^\frac{\alpha}{2}u(x)$ is determined by the value of $u$ in the whole space, instead of the value near $x$.

% In this paper, we define the higher-order fractional operator $(-\Delta)^{\frac{\alpha}{2}+m}$ as $(-\Delta)^\frac{\alpha}{2}\circ(-\Delta)^m$. Therefore, it is natural to require that $(-\Delta)^m u\in\mathcal{L}_\alpha(\mathbb{R}^n)$. Based on the integral form definition of fractional Laplacian (\ref{fraction definition}), it is clear to see the nonlocal property of fractional Laplacian.
%  For example, if $u>0$ in $\Omega$ and $u\equiv 0$ on $\Omega^c$, then for $x_0\in \Omega^c$,
% \begin{equation}
%     (-\Delta)^\frac{\alpha}{2}u(x_0)=C_{n,\alpha}P.V. \int_{\mathbb{R}^n} \frac{-u(y)}{|x-y|^{n+\alpha}}dy<0.
% \end{equation}
% Due to the non-locality of $(-\Delta)^\frac{\alpha}{2}$, we need to impose a stronger boundary condition on fractional part. This explains the boundary condition $(-\Delta)^m u=0$ on $\Omega^c$ of (\ref{in-1}).

In this paper, we introduce the higher-order fractional operator $(-\Delta)^{\frac{\alpha}{2}+m}$, which can be defined as $(-\Delta)^\frac{\alpha}{2}\circ(-\Delta)^m$. As a result, it becomes essential to ensure that $(-\Delta)^m u\in\mathcal{L}_\alpha(\mathbb{R}^n)$. 

The integral form definition of the fractional Laplacian (\ref{fraction definition}) clearly demonstrates its nonlocal property.
For instance, consider the scenario where $u>0$ in $\Omega$, and $u\equiv 0$ on $\Omega^c$. In this case, for $x_0\in \Omega^c$, we have:
\begin{equation}
(-\Delta)^\frac{\alpha}{2}u(x_0)=C_{n,\alpha}P.V. \int_{\mathbb{R}^n} \frac{-u(y)}{|x-y|^{n+\alpha}}dy<0.
\end{equation}
The non-local nature of $(-\Delta)^\frac{\alpha}{2}$ necessitates a stronger boundary condition on the fractional part. This explains the imposition of the boundary condition $(-\Delta)^m u=0$ on $\Omega^c$ in equation (\ref{in-1}).

As $m=0$ in (\ref{in-1}), the equation reduces to the following fractional Laplace equation:
\begin{equation}\label{fractional PDE}
    \left\{
    \begin{aligned}
        (-\Delta)^\frac{\alpha}{2}u(x)=u^p(x),\quad x\in\Omega,\\
        u(x)\equiv 0,\quad x\in \Omega^c.
    \end{aligned}
    \right.
\end{equation}
In \cite{chen2016direct}, Chen, Li, and Li applied a direct method on the fractional operator to derive a priori estimates of positive solutions of the equation (\ref{fractional PDE}) under the subcritical case. They obtained the following result:
\begin{equation}
\|u\|_{L^\infty(\Omega)} \leq C,
\end{equation}
where $C$ is a positive constant independent of $u$. 

Now, when dealing with the higher-order fractional Laplacian, the operator combines features of both the Laplacian and the fractional Laplacian, presenting new challenges compared to the fractional case. To address these challenges and obtain a priori estimates of positive solutions, we propose a novel approach building upon the direct method from \cite{chen2016direct}. For other types of elliptic equations and a priori estimates see \cite{li2021pointwise,chen1997priori,ruiz2004priori,amann1998priori,ros2014dirichlet,ros2015nonexistence,wang2023nonexistence,schippa2021priori,feehan2014schauder}. The main result is as follows:

% Recently, we (\cite{xu2022super}) considered the higher-order fractional Laplace equation, and obtained the integral representation formula of solutions. 
% Inspired by \cite{chen2016direct}, we investigate the higher-order fractional Laplace equation (\ref{in-1}) and get a priori estimates of positive solutions.

\begin{theorem}\label{main theorem}
Let $n\geq 2$, $0<\alpha<2$, $m\in\mathbb{N}^+$. Assume that $\Omega$ is a bounded domain with $C^2$ boundary, $1<p\leq\frac{n+\alpha}{n-2m}$, $u\in C_{loc}^{2m+[\alpha],\{\alpha\}+\epsilon}(\Omega)$, $(-\Delta)^m u\in \mathcal{L_\alpha}(\mathbb{R}^n)$, and $(-\Delta)^i u$ is upper semicontinuous on $\overline{\Omega}$ for $i=0,1,...,m$, $\epsilon>0$ sufficiently small. Suppose $u$ is a positive solution of (\ref{in-1}).
%\begin{equation*}\label{main equation}
%\left\{
%    \begin{aligned}
 % &(-\Delta)^{\frac{\alpha}{2}}(-%\Delta)^m u(x)=u^p(x),\quad x\in \Omega,\\
%  &(-\Delta)^i u= 0,\quad x\in \partial\Omega,\quad i=0,1,...,m-1,\\
 % &(-\Delta)^m u \equiv 0,\quad x\in \Omega^c.
%\end{aligned}
%\right.
%\end{equation*}

Then
\begin{equation}\label{main conclusion}
    \|u\|_{L^\infty(\Omega)}\leq C,
\end{equation}
where $C$ is a positive constant independent of $u$.
\end{theorem}
\begin{remark}
In Theorem \ref{main theorem}, combining with nonnegative property of $(-\Delta)^i u$ derived from the {\itshape maximum principle} (see Theorem 1.1 in \cite{liu2022dirichlet} for $i=m$),
the  upper 
semicontinuity of $(-\Delta)^i u$ implies 
the continuity of  $(-\Delta)^i u$ on $\Bar{\Omega}$.
\end{remark}

% For fractional equation (\ref{fractional PDE}), the basic method is applying a transformation on the domain $\Omega$ and discuss it in three cases and making H\"older estimates respectively. So far there is no result about a priori estimates of high-order fractional Laplacian. Compared with (\ref{fractional PDE}), the high-order fractional equation (\ref{main equation}) has integer parts. So we divide the equation into a system with every single equation being fractional equation or one-order equation and use H\"older estimates not only on fractional Poisson equations but also integer Poisson equations. Our strategy hinges on finding a proper parameter to transform the domain, and the parameter is expected to give a uniform bound for all equations in the system. The parameter is the core of our strategy.

For the fractional equation (\ref{fractional PDE}), the standard approach involves applying a domain transformation and then analyzing the resulting equation in three cases using H\"older estimates. However, there have been no a priori estimates for higher-order fractional Laplacians. In comparison to the fractional equation (\ref{fractional PDE}), the higher-order fractional Laplacian consists of higher-order Laplacian and fractional Laplacian. Hence both the difficulties of higher-order and those of fractional problems will arise. To tackle these, we divide the equation into a system in which one equation is a fractional equation and others are first-order equations, and apply H\"older estimates not only to the fractional Poisson equation but also to the integer Poisson equations. Our strategy relies on identifying a suitable parameter to transform the domain.
% one that can provide a uniform bound for all the equations in the system. This parameter is the cornerstone of our approach.
Unlike previous works where the parameter is chosen as the upper bound of a single equation, we select a parameter that bounds all equations in the system uniformly. This parameter is the core of our strategy, which enables us to apply H\"older estimates not only to fractional Poisson equation but also to integer Poisson equations. In the process of blowing-up, the application of maximum principles is essential. Maximum principles play a crucial role in understanding various properties of the solutions. For different types of maximum principles see \cite{li2018maximum,liu2022dirichlet,li2023maximum}.
% In some crucial estimates we apply the {\itshape maximum principle} in \cite{liu2022dirichlet}. There are various types of maximum principles for fractional Laplacian. For example, in \cite{li2023maximum} the authors studied maximum principles with critical integrability.

We also derive that every $(-\Delta)^i u$ is $L^\infty$-bounded, $i=1,...,m$.
\begin{corollary}\label{corollary}
Let $n\geq 2$, $0<\alpha<2$, $m\in\mathbb{N}^+$. Assume that $\Omega$ is a bounded domain with $C^2$ boundary, $1<p\leq\frac{n+\alpha}{n-2m}$, $u\in C_{loc}^{2m+[\alpha],\{\alpha\}+\epsilon}(\Omega)$, $(-\Delta)^m u\in \mathcal{L_\alpha}(\mathbb{R}^n)$, and $(-\Delta)^i u$ is upper semicontinuous on $\overline{\Omega}$ for $i=0,1,...,m$, $\epsilon>0$ sufficiently small. Suppose $u$ is a positive solution of the (\ref{in-1}), then we have
\begin{equation}
    \|(-\Delta)^i u\|_{L^\infty(\Omega)}\leq C_i,
\end{equation}
for $i=1,...,m$.
\end{corollary}

Similar to (\ref{in-1}), Chen and Wu \cite{chen2021uniform} studied the following higher critical
order fractional equation:
\begin{equation*}
\left\{
    \begin{aligned}
  &(-\Delta)^{\frac{n}{2}} u(x)=u^p(x),\quad x\in \Omega,\quad\Omega\mbox{ bounded domain},\\
  &(-\Delta)^i u= 0,\quad x\in \Omega^c,\quad i=0,1,...,\tfrac{n-1}{2} .
\end{aligned}
\right.
\end{equation*}
 \cite{chen2021uniform} yielded a uniform $L^p$ estimate, which states that there exists $p_0>1$, such that for any $1<p_0\leq p<\infty$,
\begin{equation*}
    \int_\Omega u^p(x)dx <C(p_0,n,\Omega).
\end{equation*}
In this paper, we consider more general equation by providing $L^\infty$ estimates under weaker boundary conditions. As such, our study provides a more complete understanding of the behavior of the solution in different regimes of $p$.

For the last few years, there are many researches about a priori estimates for Laplacian equation.
% Brezis and Turner (see \cite{brezis1977class}) first focused on elliptic equation
% \begin{equation}\label{general elliptic eq}
%     \left\{
%     \begin{aligned}
%        L u= g(x,u,Du)\;\;\mbox{in }\Omega\\
%         u=0\;\;\mbox{on }\partial\Omega,
%     \end{aligned}
%     \right.
% \end{equation}
% where $L$ is a linear elliptic operator and $g$ has some growth conditions, and proved
% \begin{equation}
%     \|u\|_{L^\infty(\Omega)}\leq C.
% \end{equation}
% Then they applied a priori estimate to get the existence results of (\ref{general elliptic eq}).
In \cite{gidas1981priori}, Gidas and Spruk used the blowing-up argument and gave a priori estimate on the boundary problem
\begin{equation}
    \left\{
    \begin{aligned}
       \partial_j(a^{ij}\partial_i u)+ b_j\partial_j u + f(x,u)=0\;\;\;\;\mbox{in }\Omega\\
        u=\phi\;\;\mbox{on }\partial\Omega
    \end{aligned}
    \right.
\end{equation}
with the growth condition
\begin{equation}\label{growth condition}
    \lim_{t\to+\infty} \frac{f(x,t)}{t^p} =h(x)\;\;\;\;\mbox{uniformly for }x\in\Bar{\Omega}.
\end{equation}
Here $1<p<\frac{n+2}{n-2}$, $h(x)$ is continuous and strictly positive in $\Bar{\Omega}$, and $\Omega$ is a bounded domain with $C^1$-boundary. One may notice that condition (\ref{growth condition}) is satisfied, for example, if $f(x,u)=u^p$.
% Using a priori estimates $|u(x)|\leq C$, they also gave the non-existence of positive solution of
% \begin{equation}
%     -\Delta u=u^p\;\;\;\;\mbox{ in }\mathbb{R}^n.
% \end{equation}
Later in \cite{de2013priori}, de Figueiredo, Lions and Nussbaum gave a priori estimates on similinear linear equation
\begin{equation}
    \left\{
    \begin{aligned}
        -\Delta u= f(u)\;\;\mbox{in }\Omega\\
        u=0\;\;\mbox{on }\partial\Omega,
    \end{aligned}
    \right.
\end{equation}
with $f$ satisfies
\begin{equation}
    \lim_{t\to+\infty}\frac{f(t)}{t^{(n+2)/(n-2)}}=0
\end{equation}
and some other technical assumptions. Here $\Omega$ is a bounded domain with $C^{2,\alpha}$-boundary. Then they obtained the existence results. But if $f(t)$ grows faster than $t^{(n+2)/(n-2)}$, then a priori estimates may fail, see \cite{joseph1973quasilinear}.

Recently, some results on fractional Laplacian, including the paper \cite{chen2016direct} we follow, are obtained. In \cite{chen2016direct} a priori estimates and topological degree arguments are used to derive existence of solutions of (\ref{fractional PDE}). In \cite{zhang2023priori}, Zhang and Cheng improved the results by giving a priori estimates near the boundary of positive solutions of
\begin{equation}
    \left\{
    \begin{aligned}
        (-\Delta)^\frac{\alpha}{2} u&= u^{(n+\alpha)/(n-\alpha)}\;\;\mbox{in }\Omega\\
        u&=0\;\;\mbox{on }\Omega^c.
    \end{aligned}
    \right.
\end{equation}
Explicitly, they showed that there exists $\delta>0$ and constant $C$ (independent of $u$) such that $\|u\|_{L^\infty(\Omega_\delta)}\leq C$ with $\Omega_\delta=\{x\mid dist(x,\partial\Omega)\leq\delta\}$. 
% In \cite{li2021pointwise} 

When proving Theorem \ref{main theorem}, %\textcolor{red}{after blowing up a sequence of solutions},  
the maxima of a sequence of solutions
may blow up on various locations of the rescaling domain, and hence we need to apply the Liouville theorems of higher-order fractional Laplacians both in the whole space and on a half space to derive contradictions. The Liouville theorem for the whole space is from the previous work \cite{xu2022super}. However, currently, there are no available results for the Liouville theorem on a half space under the locally bounded assumption. Although in \cite{zhuo2019liouville} (Theorem 1.1) a Liouville theorem was established on a half space, it requires the condition that the solution is globally integrable, which cannot be satisfied by the solution after blowing-up. Therefore, in this paper, we first give the Liouville theorem of the corresponding integral equation on a half space, and then we combine it with an equivalence result to get the Liouville theorems of higher-order fractional Laplacians on a half space. First we present the following theorem.
\begin{theorem}\label{moving planes}
   Assume $m$ is an integer, $0<\alpha<2$, $2s=2m+\alpha<n$, $1<p\leq\frac{n+\alpha}{n-2m}$. If $u(x)$ is a locally bounded positive solution for 
\begin{equation}\label{equiv integral equation}
    u(x)=\int_{\mathbb{R}^n_+}G_{2s}^+(x,y)u^p(y)dy.
\end{equation}
    $u$ must be trivial.
\end{theorem}
Here $G_{2s}^+(x,y)$ represents the Green's function of $(-\Delta)^s$, and its explicit expression will be provided in  Section \ref{sec 2}. Combining Theorem \ref{moving planes} with the equivalence between integral equation (\ref{equiv integral equation}) and differential equation  (\ref{intro eq 2}) (refer to \cite{li2016symmetry}, Theorem 4.1), we have
\begin{corollary}\label{lm}
Assume $m$ is an integer, $0<\alpha<2$, $2s=2m+\alpha<n$, $1<p\leq\frac{n+\alpha}{n-2m}$, $u\in C_{loc}^{2m+[\alpha],\{\alpha\}+\epsilon}(\mathbb{R}^n_+)$ is a positive solution of the following equation
\begin{equation}\label{intro eq 2}
\left\{
\begin{aligned}
    &(-\Delta)^\frac{\alpha}{2}(-\Delta)^m u(x)=u^p(x),\quad x\in \mathbb{R}^n_+,\\
    &(-\Delta)^i u=0\mbox{ on }\partial \mathbb{R}^n_+\mbox{ for }i=0,1,...,m-1,\\
    &(-\Delta)^m u=0\mbox{ in } (\mathbb{R}^n_+)^c.
    \end{aligned}
    \right.
\end{equation}
If $u$ is locally bounded and $(-\Delta)^i u\geq 0$ in $\mathbb{R}^n_+$ for $i=1,2,...,m-1$, then $u$ can only be trivial, i.e. $u\equiv 0$.    
\end{corollary}
% \begin{equation}\label{PDE in equiv thm}
% \left\{
% \begin{aligned}
%     &(-\Delta)^\frac{\alpha}{2}(-\Delta)^m u(x)=u^p(x),\quad x\in \mathbb{R}^n_+,\\
%     &(-\Delta)^i u=0\mbox{ on }\partial \mathbb{R}^n_+\mbox{ for }i=0,1,...,m-1,\\
%     &(-\Delta)^m u=0\mbox{ in } (\mathbb{R}^n_+)^c.
%     \end{aligned}
%     \right.
% \end{equation}
\begin{remark}
    It is important to note that the upper bound of $p$ is chosen as $\frac{n+\alpha}{n-2m}$, which is smaller than the critical index $\frac{n+2m+\alpha}{n-2m-\alpha}$. This restriction on $p$ is necessary because, to establish Theorem \ref{moving planes}, we employ the Kelvin transform in the moving plane method to handle the issue that $u$ lacks global integrability. Ensuring the positivity of the index for the singularity (see (\ref{eq 10})) is essential in this approach. In our future work, we hope to reach the critical index for $p$ or even remove the upper bound for $p$.
\end{remark}

It is well known that the Liouville type theorems are crucial in proving a priori estimates of the solutions. The Liouville type theorems
for various PDE and IE systems have been extensively studied by many authors. For higher-order Laplacians on a half space, refer to \cite{cao2013liouville}; for fractional Laplacians on a half space, refer to \cite{chen2015liouville}; for higher-order fractional Laplacians on a half space, refer to \cite{zhuo2019liouville}; and for systems of fractional Laplacians on a half space, refer to \cite{dai2017liouville}. For various Liouville type theorems, see \cite{cheng2016liouville,fang2012liouville,cao2021super,cao2021liouville}.

This paper is structured as follows: In Section \ref{sec 2}, we will prove Theorem \ref{moving planes} and Corollary \ref{lm}. In Section \ref{sec 3}, we will prove Theorem \ref{main theorem} and Corollary \ref{corollary}.

\section{The Liouville theorem of higher-order fractional Laplacians on a half space}\label{sec 2}

Denote 
\begin{equation*}
    G_{2s}^+(x,y)=\int_{\mathbb{R}^n_+}G_{2m}^+(x,z)G_\alpha^+(z,y)dz,
    \end{equation*}
where $G_{2m}^+(x,y)$ as the Green's function of $(-\Delta)^{m}$ in $\mathbb{R}^n_+$, and $G_{\alpha}^+(x,y)$ as the Green's function of $(-\Delta)^{\frac{\alpha}{2}}$ in $\mathbb{R}^n_+$, i.e.
\begin{equation*}
    G_{2m}^+(x,y)=\begin{cases}
    \frac{1}{|x-y|^{n-2m}}-\frac{1}{|x^*-y|^{n-2m}},\;\;\;\;&x,y\in\mathbb{R}^n_+,\;x\neq y,\\
    0,\;\;\;\;&\mbox{otherwise},
    \end{cases}
\end{equation*}
where $x^*$ is the reflection of $x$ about $\partial\mathbb{R}^n_+$, and 
\begin{equation*}
    G_\alpha^+(x,y)=\begin{cases}
        \frac{C_n}{|x-y|^{n-\alpha}}\int_0^\frac{4x_ny_n}{|x-y|^2} \frac{z^{\alpha/2-1}}{(z+1)^\frac{n}{2}}dz,\;\;\;\;&x,y\in\mathbb{R}^n_+,\;x\neq y,\\
         0,\;\;\;\;&\mbox{otherwise}.
    \end{cases}
\end{equation*}
One may verify the following lemma by the direct calculation.
\begin{lemma}\label{kelvin prop of green functions}
    For $x,y\in\mathbb{R}^n_+,\;x\neq y$, $z^0\in\mathbb{R}^n$, we have
    \begin{equation*}
        \begin{aligned}
            \frac{G_{2m}^+(\frac{x-z^0}{|x-z^0|^2},\frac{y-z^0}{|y-z^0|^2})}{|x-z^0|^{n-2m}|y-z^0|^{n-2m}}=G_{2m}^+(x-z^0,y-z^0),\\
            \frac{G_{\alpha}^+(\frac{x-z^0}{|x-z^0|^2},\frac{y-z^0}{|y-z^0|^2})}{|x-z^0|^{n-\alpha}|y-z^0|^{n-\alpha}}=G_{\alpha}^+(x-z^0,y-z^0).
        \end{aligned}
    \end{equation*}
\end{lemma}

\begin{lemma}[Hardy-Littlewood-Sobolev inequality in general domains]\label{HLS lemma}
    Assume $0<\alpha<n$, $A,B\subset \mathbb{R}^n.$ Let $g\in L^{\frac{nq}{n+\alpha q}}(A)$ for $\frac{n}{n-\alpha}<q<\infty$. Then
    \begin{equation*}
        \| \int_A \frac{g(y)}{|x-y|^{n-\alpha}}dy\|_{L^q(B)} \leq C(n,q,\alpha) \|g\|_{L^\frac{nq}{n+\alpha q}(A)}.
    \end{equation*}
\end{lemma}
\begin{proof}
    By classical HLS inequality, we have 
    \begin{equation*}
        \| \int_A \frac{g(y)}{|x-y|^{n-\alpha}}dy\|_{L^q(\mathbb{R}^n)} = \| \int_{\mathbb{R}^n} \frac{g(y)\chi_A(y)}{|x-y|^{n-\alpha}}dy \| _{L^q({\mathbb{R}^n})} \leq C(n,q,\alpha) \|g\chi_A\|_{L^\frac{nq}{n+\alpha q}(\mathbb{R}^n)} =C(n,q,\alpha) \|g\|_{L^\frac{nq}{n+\alpha q}(A)}.
    \end{equation*}
    Thus 
    \begin{equation*}
        \| \int_A \frac{g(y)}{|x-y|^{n-\alpha}}dy\|_{L^q(B)} \leq    \| \int_A \frac{g(y)}{|x-y|^{n-\alpha}}dy\|_{L^q(\mathbb{R}^n)} \leq C(n,q,\alpha) \|g\|_{L^\frac{nq}{n+\alpha q}(A)}.
    \end{equation*}
\end{proof}

    We use the method of moving planes to prove that (\ref{equiv integral equation}) admits only trivial solutions. Before proving Theorem \ref{moving planes}, we state some notations and properties. Denote 
\begin{equation*}
    u_m(x)=\int_{\mathbb{R}^n_+}G_\alpha^+(x,y)u^p(y)dy.
\end{equation*}
By definition of $G_{2s}^+(x,y)$, one can see that
\begin{equation}\label{integral u}
    u(x)=\int_{\mathbb{R}^n_+}G_{2m}^+(x,y)u_m(y)dy.
\end{equation}
For any $z^0\in\mathbb{R}^n$, denote
    \begin{equation*}
        \begin{aligned}
            \bar{u}(x)=\frac{1}{|x-z^0|^{n-2m}}u(\frac{x-z^0}{|x-z^0|^2}),\;\;\;\;x\neq z^0,\\
            \bar{u}_m(x)=\frac{1}{|x-z^0|^{n-\alpha}}u_m(\frac{x-z^0}{|x-z^0|^2})\;\;\;\;x\neq z^0.
        \end{aligned}        
    \end{equation*}
By a translation, we can always assume $z^0=0$.
For $\lambda\in\mathbb{R}$, denote
\begin{equation}
    \begin{aligned}
        \Sigma_\lambda=\mathbb{R}^n_+\cap\{x=(x_1,\cdots,x_n)\in\mathbb{R}_n \mid x_1<\lambda\},\\
        x^\lambda=(2\lambda-x_1,x_2,\cdots,x_n), \\
         \Bar{u}^\lambda(x)=\Bar{u}(x^\lambda),\;\;\Bar{u}^\lambda_m(x)=\Bar{u}_m(x^\lambda).
    \end{aligned}
\end{equation}
Denote
\begin{equation*}
    \begin{aligned}
        w_m^\lambda(x)=\Bar{u}_m^\lambda(x)-\Bar{u}_m(x),\\
        w^\lambda(x)=\Bar{u}^\lambda(x)-\Bar{u}(x).
    \end{aligned}
\end{equation*}

The following Lemma is the direct results of symmetry.
\begin{lemma}\label{sym prop of green functions}
     For $x,y\in\mathbb{R}^n_+,\;x\neq y$, we have
     \begin{equation*}
         \begin{aligned}
             G_\alpha^+(x,y)=G_\alpha^+(x^\lambda,y^\lambda),\;G_\alpha^+(x^\lambda,y)=G_\alpha^+(x,y^\lambda),\\
             G_{2m}^+(x,y)=G_{2m}^+(x^\lambda,y^\lambda),\;G_{2m}^+(x^\lambda,y)=G_{2m}^+(x,y^\lambda).
         \end{aligned}
     \end{equation*}
\end{lemma}
\begin{proof}[\textbf{Proof of Theorem \ref{moving planes}}]
       There are four steps of the proof.

\emph {Step 1.} Suppose $u$ is not trivial. For $\lambda$ sufficiently negative, we will show that $w_m^\lambda\geq 0,\; w^\lambda\geq 0$ a.e in $\Sigma_\lambda$. This provides a starting point to move the plane $\{x\in\mathbb{R}^n_+\mid x_1=\lambda\}$.

By (\ref{integral u}),
\begin{equation*}
    \begin{aligned}
        \Bar{u}(x)=&\frac{1}{|x|^{n-2m}}\int_{\mathbb{R}^n_+}G_{2m}^+(\frac{x}{|x|^2},y)u_m(y)dy\\
        =&\frac{1}{|x|^{n-2m}}\int_{\mathbb{R}^n_+}G_{2m}^+(\frac{x}{|x|^2},\frac{z}{|z|^2})u_m(\frac{z}{|z|^2})\frac{dz}{|z|^{2n}}\\
        =&\int_{\mathbb{R}^n_+} \frac{ G_{2m}^+(\frac{x}{|x|^2},\frac{z}{|z|^2})}{|x|^{n-2m}|z|^{n-2m}}
       u_m(\frac{z}{|z|^2})\frac{dz}{|z|^{n+2m}}\\
    \overset{Lemma\;\ref{kelvin prop of green functions}}{=} &\int_{\mathbb{R}^n_+}G_{2m}^+(x,z)\frac{u_m(\frac{z}{|z|^2})}{|z|^{n-\alpha}}\frac{dz}{|z|^{2m+\alpha}}\\
    =&\int_{\mathbb{R}^n_+}G_{2m}^+(x,z)\frac{\Bar{u}_m(z)}{|z|^{\beta_2}}dz,\;\;\;\;\beta_2=2m+\alpha>0.
    \end{aligned}
\end{equation*}
Similarly, we have
\begin{equation}\label{eq 10}
    \Bar{u}_m(x)=\int_{\mathbb{R}^n_+}G_\alpha^+(x,z)\frac{\bar{u}^p(z)}{|z|^{\beta_1}}dz,\;\;\;\; \beta_1=n+\alpha-p(n-2m)\geq0.
\end{equation}

Using Lemma \ref{sym prop of green functions}, we get
\begin{equation*}
\begin{aligned}
       \Bar{u}_m(x)=\int_{\Sigma_\lambda}G_\alpha^+(x,y)\frac{\bar{u}^p(y)}{|y|^{\beta_1}}dy + \int_{\Sigma_\lambda}G_\alpha^+(x,y^\lambda)\frac{(\bar{u}^\lambda(y))^p}{|y^\lambda|^{\beta_1}}dy ,\\
       \Bar{u}_m^\lambda(x)=\int_{\Sigma_\lambda}G_\alpha^+(x^\lambda,y)\frac{\bar{u}^p(y)}{|y|^{\beta_1}}dy + \int_{\Sigma_\lambda}G_\alpha^+(x,y)\frac{(\bar{u}^\lambda(y))^p}{|y^\lambda|^{\beta_1}}dy.
\end{aligned}
\end{equation*}
So
\begin{equation}\label{eq 1}
    \Bar{u}_m(x)-\Bar{u}_m^\lambda(x)=\int_{\Sigma_\lambda} \left(G_\alpha^+(x,y)-G_\alpha^+(x^\lambda,y)\right)\left(\frac{\bar{u}^p(y)}{|y|^{\beta_1}}-\frac{(\bar{u}^\lambda(y))^p}{|y^\lambda|^{\beta_1}} \right)dy.
\end{equation}
Similarly,
\begin{equation}\label{eq 1-1}
        \Bar{u}(x)-\Bar{u}^\lambda(x)=\int_{\Sigma_\lambda} \left(G_{2m}^+(x,y)-G_{2m}^+(x^\lambda,y)\right) \left(\frac{\bar{u}_m(y)}{|y|^{\beta_2}}-\frac{\bar{u}^\lambda_m(y)}{|y^\lambda|^{\beta_2}} \right)dy.
\end{equation}

Define 
\begin{equation*}
    \begin{aligned}
        \Sigma_\lambda^- =\{ x\in\Sigma_\lambda\backslash B_\xi(0^\lambda)\mid w^\lambda<0 \},\\
        \Sigma_{m,\lambda}^- =\{ x\in\Sigma_\lambda\backslash B_\xi(0^\lambda)\mid w_m^\lambda<0 \}.
    \end{aligned}
\end{equation*}

Here $\xi>0$ is sufficiently small, $0^\lambda=(2\lambda,0,\cdots,0)$. For any $x\in \Sigma_{m,\lambda}^-$, by (\ref{eq 1}) it follows that
\begin{equation}\label{eq 2}
    \begin{aligned}
        0<-w_m^\lambda=&\Bar{u}_m(x)-\Bar{u}_m^\lambda(x)\\
        =& \int_{\Sigma_\lambda^-} \left(G_\alpha^+(x,y)-G_\alpha^+(x^\lambda,y)\right)\left(\frac{\bar{u}^p(y)}{|y|^{\beta_1}}-\frac{(\bar{u}^\lambda(y))^p}{|y^\lambda|^{\beta_1}} \right)dy\\
    &+\int_{\Sigma_\lambda\backslash\Sigma_\lambda^-} \left(G_\alpha^+(x,y)-G_\alpha^+(x^\lambda,y)\right)\left(\frac{\bar{u}^p(y)}{|y|^{\beta_1}}-\frac{(\bar{u}^\lambda(y))^p}{|y^\lambda|^{\beta_1}} \right)dy\\
    \leq &  \int_{\Sigma_\lambda^-} \left(G_\alpha^+(x,y)-G_\alpha^+(x^\lambda,y)\right)\left(\frac{\bar{u}^p(y)}{|y|^{\beta_1}}-\frac{(\bar{u}^\lambda(y))^p}{|y^\lambda|^{\beta_1}} \right)dy\\
    = & \int_{\Sigma_\lambda^-} \left(G_\alpha^+(x,y)-G_\alpha^+(x^\lambda,y)\right)\left( \frac{\bar{u}^p(y)}{|y|^{\beta_1}}- \frac{(\bar{u}^\lambda(y))^p}{|y|^{\beta_1}} +\frac{(\bar{u}^\lambda(y))^p}{|y|^{\beta_1}} - \frac{(\bar{u}^\lambda(y))^p} {|y^\lambda|^{\beta_1}} \right) dy\\
    \leq &\int_{\Sigma_\lambda^-} \left(G_\alpha^+(x,y)-G_\alpha^+(x^\lambda,y)\right) \left( \frac{\bar{u}^p(y)}{|y|^{\beta_1}}- \frac{(\bar{u}^\lambda(y))^p}{|y|^{\beta_1}}  \right) dy\\
    \overset{\mbox{MVT}}{=}&p \int_{\Sigma_\lambda^-} \left(G_\alpha^+(x,y)-G_\alpha^+(x^\lambda,y)\right) \frac{\psi_\lambda^{p-1}}{|y|^{\beta_1}}(\Bar{u}(y)-\Bar{u}^\lambda(y))dy,\;\;\;\; \text{where } \Bar{u}^\lambda(y)\leq \psi_\lambda(y)\leq \bar{u}(y),\\
    \leq &p \int_{\Sigma_\lambda^-} \left(G_\alpha^+(x,y)-G_\alpha^+(x^\lambda,y)\right) \frac{\bar{u}^{p-1}(y)}{|y|^{\beta_1}}(\Bar{u}(y)-\Bar{u}^\lambda(y))dy,\\
    \leq & C(p,n,\alpha) \int_{\Sigma_\lambda^-} \frac{1}{|x-y|^{n-\alpha}}\frac{\Bar{u}^{p-1}(y)}{|y|^{\beta_1}}|w^\lambda(y)|dy.
    \end{aligned}
\end{equation}
The last inequality is because
\begin{equation*}
    G_\alpha^+(x,y)\leq \frac{C(n,\alpha)}{|x-y|^{n-\alpha}},\;\;\;\;x\neq y.
\end{equation*}
Let $q$ be a real number greater than $\frac{n}{n-2m}$. By (\ref{eq 2}), Lemma \ref{HLS lemma} and H\"older inequality, we have
\begin{equation*}
\begin{aligned}
        \|w_m^\lambda\|_{L^q(\Sigma_{m,\lambda}^-)} \leq &\| C(p,n,\alpha) \int_{\Sigma_\lambda^-} 
 \frac{1}{|x-y|^{n-\alpha}} \frac{\Bar{u}^{p-1}(y)}{|y|^{\beta_1}}|w^\lambda(y)| dy\|_{L^q(\Sigma_{m,\lambda}^-)} \\
 \leq & C(p,q,n,\alpha)\|  \frac{\Bar{u}^{p-1}(y)}{|y|^{\beta_1}} w^\lambda(y)  \|_{L^\frac{nq}{n+\alpha q}(\Sigma_{\lambda}^-)}\\
 \leq & C(p,q,n,\alpha) \|w^\lambda\| _{L^q(\Sigma_{\lambda}^-)} \| \frac{\Bar{u}^{p-1}(y)}{|y|^{\beta_1}}  \|_{L^\frac{n}{\alpha}(\Sigma_{\lambda}^-)} .
\end{aligned}    
\end{equation*}
Similarly,
\begin{equation*}
\begin{aligned}
\|w^\lambda\|_{L^q(\Sigma_{\lambda}^-)} \leq &\| C(p,n,m) \int_{\Sigma_{m,\lambda}^-} 
 \frac{1}{|x-y|^{n-2m}} \frac{w_m^\lambda(y)}{|y|^{\beta_2}} dy\|_{L^q(\Sigma_\lambda^-)} \\
 \leq & C(p,q,n,m)\|  \frac{w_m^\lambda(y)}{|y|^{\beta_2}}   \|_{L^\frac{nq}{n+2m q}(\Sigma_{m,\lambda}^-)}\\
 \leq & C(p,q,n,m) \|w_m^\lambda\| _{L^q(\Sigma_{m,\lambda}^-)} \| \frac{1}{|y|^{\beta_2}}  \|_{L^\frac{n}{2m}(\Sigma_{m,\lambda}^-)} .
\end{aligned}    
\end{equation*}
Combining these two inequalities together, it follows that
\begin{equation}\label{eq 3}
\begin{aligned}
     \|w_m^\lambda\|_{L^q(\Sigma_{m,\lambda}^-)} \leq C(n,p,q,m,\alpha)  \|w_m^\lambda\|_{L^q(\Sigma_{m,\lambda}^-)}  \|\frac{\Bar{u}^{p-1}(y)}{|y|^{\beta_1}}  \|_{L^\frac{n}{\alpha}(\Sigma_{\lambda}^-)}  \|\frac{1}{|y|^{\beta_2}}  \|_{L^\frac{n}{2m}(\Sigma_{m,\lambda}^-)} ,\\
 \|w^\lambda\|_{L^q(\Sigma_{\lambda}^-)} \leq C(n,p,q,m,\alpha) \|w^\lambda\|_{L^q(\Sigma_{\lambda}^-)} \| \frac{\Bar{u}^{p-1}(y)}{|y|^{\beta_1}}  \|_{L^\frac{n}{\alpha}(\Sigma_{\lambda}^-)} \|\frac{1}{|y|^{\beta_2}}  \|_{L^\frac{n}{2m}(\Sigma_{m,\lambda}^-)} .
\end{aligned}
\end{equation}
Local boundness of $u$ implies that
\begin{equation*}
    \Bar{u}(x) \sim \frac{1}{|x|^{n-2m}},\;\;\;\;|x|\to\infty.
\end{equation*}
It follows that
\begin{equation*}
    \left(\frac{\Bar{u}^{p-1}(y)}{|y|^{\beta_1}}\right)^\frac{n}{\alpha}\sim \frac{1}{|y|^\frac{n(2m+\alpha)}{\alpha}} ,\;\;\;\;|y|\to\infty.
\end{equation*}
Also note that 
\begin{equation*}
    \begin{aligned}
       \left( \frac{1}{|y|^{\beta_2}}\right)^\frac{n}{2m}=\frac{1}{|y|^\frac{n(2m+\alpha)}{2m}} .
    \end{aligned}
\end{equation*}
Thus we can choose $\lambda$ sufficiently negative, such that
\begin{equation}\label{eq 4}
    \begin{aligned}
        C(n,p,q,m,\alpha) \| \frac{\Bar{u}^{p-1}(y)}{|y|^{\beta_1}}  \|_{L^\frac{n}{\alpha}(\Sigma_{\lambda}^-)} \|\frac{1}{|y|^{\beta_2}}  \|_{L^\frac{n}{2m}(\Sigma_{m,\lambda}^-)} \leq \frac{1}{2}.
    \end{aligned}
\end{equation}
(\ref{eq 3}) and (\ref{eq 4}) imply
\begin{equation*}
     \|w_m^\lambda\|_{L^q(\Sigma_{m,\lambda}^-)} =0,\;  \|w^\lambda\|_{L^q(\Sigma_{\lambda}^-)}=0,\;\;\;\;\text{for $\lambda$ sufficiently negative}.
\end{equation*}
Thus 
\begin{equation*}
    |\Sigma_{m,\lambda}^-|=0,\;|\Sigma_{\lambda}^-|=0,\;\;\;\;\text{for $\lambda$ sufficiently negative}.
\end{equation*}

\emph{Step 2.} We start form the neighborhood of $x_1=-\infty$ and move the plane to the right as long as $w_m^\lambda\geq 0,\; w^\lambda\geq 0$ a.e in $\Sigma_\lambda$ holds to the limiting position. 

Define
\begin{equation*}
    \lambda_0=\sup\{ \lambda<0\mid w_m^\lambda\geq 0, w^\lambda\geq 0, \rho\leq \lambda,\forall x\in\Sigma_\rho\backslash B_\xi(0^{\rho}) \}.
\end{equation*}
We claim $\lambda_0=0$. In fact, if $\lambda_0<\epsilon _0<0$ for some $\epsilon _0>0$, we can show by contradiction that $w_m^{\lambda_0}\equiv 0 $ or $w^{\lambda_0}\equiv 0$ a.e in $\Sigma_{\lambda_0}\backslash B_\xi (0^{\lambda_0})$.  

Otherwise, both $w_m^{\lambda_0}$ and $w^{\lambda_0}\not\equiv 0$ a.e in $\Sigma_{\lambda_0}\backslash B_\xi (0^{\lambda_0})$. We will show that one can choose $\epsilon>0$ sufficiently small, so that for all $\lambda\in(\lambda_0,\lambda_0+\epsilon)$, 
\begin{equation}\label{more lmd}
    w_m^\lambda\geq 0,\;w^\lambda\geq 0\;\text{a.e. in}\;\Sigma _\lambda,
\end{equation}
which is contradicted to the definition of $\lambda _0$. Then we arrive that $w_m^{\lambda_0}\equiv 0 $ or $w^{\lambda_0}\equiv 0$ a.e in $\Sigma_{\lambda_0}\backslash B_\xi (0^{\lambda_0})$.

To verify (\ref{more lmd}), we can use the same argument in step 1 and only to show that for all $\lambda\in(\lambda_0,\lambda_0+\epsilon)$, 
$\int_{\Sigma_\lambda^-}\left( \frac{\Bar{u}^{p-1}(y)}{|y|^{\beta_1}}\right) ^\frac{n}{\alpha} dy$ is also sufficiently small.

 Since \[\left( \frac{\Bar{u}^{p-1}(y)}{|y|^{\beta_1}}\right) ^\frac{n}{\alpha}\sim \dfrac{1}{|y|^{\frac{n(2m+\alpha)}{\alpha}}},\;\text{as}\;|y|\to \infty,\]
for any $\eta >0$, we can choose $R\gg 0$, such that 
\[
\int _{\big(\mathbb{R}^n_+\backslash B_\xi(0^\lambda)\big)\backslash B_R}\left( \frac{\Bar{u}^{p-1}(y)}{|y|^{\beta_1}}\right) ^\frac{n}{\alpha}dy<\eta.
\]
To show the integral on $B_R\cap \Sigma _\lambda ^-$ is small, we only need to prove that the measure of $B_R\cap \Sigma _\lambda ^-$ is small as $\epsilon$ is sufficiently small.

We claim that $w^{\lambda _0}>0$ in the interior of $\Sigma _{\lambda _0}\backslash B_\xi(0^{\lambda _0})$. If the claim is false, one can find $x_0\in \Sigma _{\lambda _0}\backslash B_\xi(0^{\lambda _0})$, such that 
\[0=\Bar{u}(x_0)-\Bar{u}^{\lambda _0}(x_0)=\int_{\Sigma_{\lambda_0}} \left(G_{2m}^+(x_0,y)-G_{2m}^+(x^{\lambda _0}_0,y)\right) \left(\frac{\bar{u}_m(y)}{|y|^{\beta_2}}-\frac{\bar{u}^{\lambda _0}_m(y)}{|y^{\lambda _0}|^{\beta_2}} \right)dy. \]
Recalling that $G_{2m}^+(x,y)-G_{2m}^+(x^{\lambda _0},y)>0$ for any $y\in \Sigma _{\lambda _0}$, we have that 
$\frac{\bar{u}_m(y)}{|y|^{\beta_2}}=\frac{\bar{u}^{\lambda _0}_m(y)}{|y^{\lambda _0}|^{\beta_2}}$ for any $y\in \Sigma _{\lambda _0}$. This is impossible, since $w_m^{\lambda_0} \not\equiv 0$ a.e. in $\Sigma _{\lambda_0}$ and $|y|>|y^{\lambda _0}|$. Thus
\begin{equation}\label{w>0}
    w^{\lambda_0}(x)>0,\;\text{in the interior of}\;\Sigma _{\lambda _0}\backslash B_\xi(0^{\lambda _0}).
\end{equation}
Denote
\[
E_\gamma = \{x\in\big(\Sigma _{\lambda _0}\backslash B_\xi(0^{\lambda _0})\big)\cap B_R | w^{\lambda _0}>\gamma\},\; F_\gamma = \big(\Sigma _{\lambda _0}\backslash B_\xi(0^{\lambda _0})\big)\cap B_R\backslash E_\gamma.
\]
According to (\ref{w>0}), $|F_\gamma|\to 0$ as $\gamma\to 0$. Denote
\[D_\lambda =\big(\Sigma _{\lambda }\backslash B_\xi(0^{\lambda })\big)\backslash\big(\Sigma _{\lambda _0}\backslash B_\xi(0^{\lambda _0})\big)\cap B_R. \]
Then 
\[
\Sigma_{\lambda}^-\cap B_R\subset (\Sigma_{\lambda}^-\cap E_\gamma)\cup F_\gamma\cup D_\lambda.
\]
Obviously, $\lim\limits _{\epsilon\to 0}|D_\lambda|=0$. Now we prove that $|\Sigma_{\lambda}^-\cap E_\gamma|$ can be small enough if we choose $\epsilon$ is small.

For any $x\in \Sigma_{\lambda}^-\cap E_\gamma$, 
\[
0\geq w^\lambda (x)=\bar{u}^\lambda (x)-\bar{u}(x)=\bar{u}^\lambda (x)-\bar{u}^{\lambda_0} (x)+\bar{u}^{\lambda_0} (x)-\bar{u}(x).
\]
Combining with the definition of $E_\gamma$, we get $\bar{u}^{\lambda_0} (x)-\bar{u}^{\lambda} (x)>\gamma$. Thus
\[
\Sigma_{\lambda}^-\cap E_\gamma\subset G_\gamma:=\{x\in B_R\backslash \big( B_\xi(0^{\lambda})\cup B_\xi(0^{\lambda_0}) \big) \mid \bar{u}^{\lambda_0}(x)-\bar{u}^{\lambda} (x)>\gamma\}.
\]
By Chebyshev's inequality, 
\[
|G_\gamma|\leq \dfrac{1}{\gamma}\int _{G_\gamma}|\bar{u}^\lambda (x)-\bar{u}^{\lambda_0} (x)|dx.
\]
We can fix $\gamma >0$, such that $|F_\gamma|$ is small enough, then we can choose $\epsilon$ is small enough, such that $|D_\gamma|$ and $|G_\gamma|$ are small. Therefore we have proved that the measure of $B_R\cap \Sigma _\lambda^-$ is small as $\epsilon$ is small enough.

So we can choose $\epsilon>0$ sufficiently small, so that for all $\lambda\in(\lambda_0,\lambda_0+\epsilon)$, $\int_{\Sigma_\lambda^-}\left( \frac{\Bar{u}^{p-1}(y)}{|y|^{\beta_1}}\right) ^\frac{n}{\alpha} dy$ is also sufficiently small. Since it holds that
\begin{equation*}
\|w_m^\lambda\|_{L^q(\Sigma_{m,\lambda}^-)} \leq C(n,p,q,m,\alpha)  \|w_m^\lambda\|_{L^q(\Sigma_{m,\lambda}^-)}  \|\frac{\Bar{u}^{p-1}(y)}{|y|^{\beta_1}}  \|_{L^\frac{n}{\alpha}(\Sigma_{\lambda}^-)}  \|\frac{1}{|y|^{\beta_2}}  \|_{L^\frac{n}{2m}(\Sigma_{m,\lambda}^-)}\;\;\;\;\text{for all }\lambda\in(\lambda_0,\lambda_0+\epsilon),
\end{equation*}
we can choose $\epsilon$ such that 
\begin{equation*}
            C(n,p,q,m,\alpha) \| \frac{\Bar{u}^{p-1}(y)}{|y|^{\beta_1}}  \|_{L^\frac{n}{\alpha}(\Sigma_{\lambda}^-)} \|\frac{1}{|y|^{\beta_2}}  \|_{L^\frac{n}{2m}(\Sigma_{m,\lambda}^-)} \leq \frac{1}{2}.
\end{equation*}
Thus $|\Sigma_{m,\lambda}^-|=0$. Then by 
\begin{equation*}
 \|w^\lambda\|_{L^q(\Sigma_{\lambda}^-)} \leq C(n,p,q,m,\alpha) \|w^\lambda\|_{L^q(\Sigma_{\lambda}^-)} \| \frac{\Bar{u}^{p-1}(y)}{|y|^{\beta_1}}  \|_{L^\frac{n}{\alpha}(\Sigma_{\lambda}^-)} \|\frac{1}{|y|^{\beta_2}}  \|_{L^\frac{n}{2m}(\Sigma_{m,\lambda}^-)},   \;\;\;\;\text{for all }\lambda\in(\lambda_0,\lambda_0+\epsilon),
\end{equation*}
we get $|\Sigma_{\lambda}^-|=0$. Hence for all $\lambda\in(\lambda_0,\lambda_0+\epsilon)$, it follows that $w^\lambda\geq0,\; w_m^\lambda\geq 0$ a.e in $\Sigma_\lambda\backslash B_\xi(0^\lambda)$. It is contradicted to the definition of $\lambda_0$. Thus for $\lambda_0$ we have $w_m^{\lambda_0}\equiv 0,\; w^{\lambda_0}\equiv 0$ a.e in $\Sigma_{\lambda_0}\backslash B_\xi (0^{\lambda_0})$.  

But $w^{\lambda_0}\equiv 0$ or $w_m^{\lambda_0}\equiv 0$ means $\Bar{u}$ or $\Bar{u}_m$ is singular at both $0$ and $0^{\lambda_0}$, which is impossible. So $\lambda_0$ can only be 0.

\emph{Step 3.} We will show that $u(x)$ only depends on $x_n$, and so is $u_m(x)$. 

By \emph{Step 2}, $w^0\geq 0$ a.e in $\Sigma_0\backslash B_\xi(0)$. Similarly, by moving the plane from $+\infty$ to the left along the $x_1$-axis, we can show that $w^0\leq 0$ a.e in $\Sigma_0\backslash B_\xi(0)$. Thus $\bar{u}$ is symmetric about the plane passing through $0$ and perpendicular to the $x_1$-axis. Since $0$ can be any point $z^0\in\mathbb{R}^n$ and $x_1$ can be any other directions perpendicular to $x_n$-axis, we have actually reached the conclusion that $\Bar{u}$ is rotationally symmetric about the line passing through $z^0$ and parallel to the $x_n$-axis. Take two point $X_1,\;X_2\in\mathbb{R}^n_+$ with the same $x_n$-coordinate, and let $z^0$ be the projection of the midpoint of $X_1$ and $X_2$ onto $\mathbb{R}^n_+$. Let $Y_1=\frac{X_1-z^0}{|X_1-z^0|^{n-2m}}+z^0$, $Y_2=\frac{X_2-z^0}{|X_2-z^0|^{n-2m}}+z^0$. One can see that $Y_1$ and $Y_2$ is rotationally symmetric about the line passing through $z^0$ and the midpoint of $X_1$ and $X_2$. Thus $\Bar{u}(Y_1)=\Bar{u}(Y_2)$. Since $|X_1-z^0|=|X_2-z^0|$, we have $u(X_1)=u(X_2)$. Hence $u$ depends only on the $x_n$ variable, i.e. $u(x)=u(x_n)$. Moreover, we also conclude that $u_m$ depends only on the $x_n$ variable.

\emph{Step 4.} We will reach a contradiction.

Choose $R$ sufficiently large and fix $x\in\mathbb{R}^n_+$ such that $\frac{3}{2}x_n>R>1$ and $G_\alpha^+(x,y)\geq C(n,\alpha)\frac{(x_ny_n)^{\alpha/2}}{|x-y|^n}$ (see \cite{chen2015liouville}). Set $x=(x',x_n),\;y=(y',y_n),\;r=|x'-y'|,\;a=|x_n-y_n|$. Then
\begin{equation*}
    \begin{aligned}
    +\infty> u_m(x_n)
    =& \int_{\mathbb{R}^n_+}G_\alpha^+(x,y)u^p(y_n)dy,\\
    \geq &\int_{\frac{3}{2}x_n}^\infty u^p(y_n) \int_{\mathbb{R}^{n-1} \backslash B_R(0)} G^+_\alpha(x,y) dy'dy_n,\\
    \geq &C(n,\alpha) \int_{\frac{3}{2}x_n}^\infty u^p(y_n) (x_ny_n)^{\alpha/2}  \int_{\mathbb{R}^{n-1} \backslash B_R(0)} \frac{1}{|x-y|^n} dy'dy_n,\\
    \geq & C(n,\alpha) x_n^{\alpha/2} \int_{\frac{3}{2}x_n}^\infty u^p(y_n) y_n^{\alpha/2} \int_R^\infty \frac{r^{n-2}}{(r^2+a^2)^{n/2}}drdy_n,\\
    \geq &C(n,\alpha) x_n^{\alpha/2} \int_{\frac{3}{2}x_n}^\infty u^p(y_n) y_n^{\alpha/2} \frac{1}{|x_n-y_n|} \int_{R/a}^\infty \frac{\tau^{n-2}}{(\tau^2+1^2)^{n/2}}d\tau dy_n,\\
    \geq & C(n,\alpha)x_n^{\alpha/2} \int_{\frac{3}{2}x_n}^\infty u^p(y_n) y_n^{\alpha/2-1} dy_n. 
    \end{aligned}
\end{equation*}
The last inequality above is because 
\begin{equation*}
    \frac{R}{a}=\frac{R}{y_n-x_n}<\frac{2R}{x_n}<3\;\;\text{and } \frac{1}{|x_n-y_n|}>\frac{1}{y_n}\;\;\;\;\text{for }y_n>\frac{3}{2}x_n.
\end{equation*}
This implies that there exists a sequence $\{y_n^{(i)}\}$ such that
\begin{equation}\label{eq 9}
    \lim_{i\to\infty} y_n^{(i)} = \infty,\; \lim_{i\to\infty}(y_n^{(i)})^{\alpha/2} u^p(y_n^{(i)})=0.
\end{equation}
Let $x=(0,x_n)\in \mathbb{R}^n_+$. Then similarly we get
\begin{equation*}
    \begin{aligned}
        +\infty>u_m(x_n) \geq C(n,\alpha)x_n^{\alpha/2} \int_0^1 u^p(y_n)  y_n^{\alpha/2} \frac{1}{|x_n-y_n|} dy_n.
    \end{aligned}
\end{equation*}
Let $x_n=2K>2$, then
\begin{equation}\label{eq 7}
    u_m(x_n)>C(n,\alpha){K^{\alpha/2}} \int_0^1 u^p(y_n) y_n^{\alpha/2} \frac{1}{{K}} dy_n> C(n,\alpha,p,u) K^{\alpha/2-1}=  C(n,\alpha,p,u) x_n^{\alpha/2-1}.
\end{equation}
On the other hand, fix $x\in\mathbb{R}^n_+$, for $R$ sufficiently large such that $R>2x_n$,
\begin{equation}\label{eq 5}
    \begin{aligned}
        +\infty>u(x_n)= &\int_{\mathbb{R}^n_+} G_{2m}^+(x,y) u_m(x,y)dy,\\
        \geq &\int_{\frac{3}{2}x_n}^\infty u_m(y_n) \int_{\mathbb{R}^{n-1}\backslash B_R(0)} \left( \frac{1}{|x-y|^{n-2m}}-\frac{1}{|x^*-y|^{n-2m}} \right) dy'dy_n,\\
         \overset{\mbox{MVT}}{\geq} &2(n-2m)x_n\int_{\frac{3}{2}x_n}^\infty u_m(y_n)(y_n-x_n) \int_{\mathbb{R}^{n-1}\backslash B_R(0)} \frac{1}{(|x'-y'|^2+|x_n+y_n|^2)^{\frac{n-2m}{2}+1}}dy'dy_n,\\
        = & 2(n-2m)x_n \int_{\frac{3}{2}x_n}^\infty u_m(y_n)(y_n-x_n)  \int_R^\infty \frac{r^{n-2}}{(a^2+r^2)^{\frac{n-2m}{2}+1}} drdy_n,\\
         = & 2(n-2m)x_n \int_{\frac{3}{2}x_n}^\infty u_m(y_n) \frac{y_n-x_n}{|x_n+y_n|^{3-2m}} \int_{R/a}^\infty \frac{\tau^{n-2}}{(1+\tau^2)^{\frac{n-2m}{2}+1}}d\tau dy_n.
    \end{aligned}
\end{equation}
For $m\geq \frac{3}{2}$, the integral $\int_{R/a}^\infty \frac{\tau^{n-2}}{(1+\tau^2)^{\frac{n-2m}{2}+1}}d\tau$ diverges, which contradicts with the finiteness of $u(x)$. Thus we only need to consider the case $m=1$. When $m=1$, it follows from (\ref{eq 5}) that
\begin{equation}\label{eq 8}
    u(x_n)\geq C(n,m) x_n \int_{\frac{3}{2}x_n}^\infty u_m(y_n)  \int_{3}^\infty \frac{\tau^{n-2}}{(1+\tau^2)^{n/2}}d\tau dy_n
    \geq C(n,m) x_n \int_{\frac{3}{2}x_n}^\infty u_m(y_n) dy_n.
\end{equation}
Combining (\ref{eq 7}) and (\ref{eq 8}), we have
\begin{equation*}
    u(x_n)\geq C(n,m) x_n\int_{\frac{3}{2}x_n}^{2x_n}y_n^{\alpha/2-1}dy_n > C(n,m,\alpha) x_n^{\alpha/2+1}.
\end{equation*}
Thus 
\begin{equation*}
    u^p(x_n) x_n^{\alpha/2}>C(n,m,\alpha) x_n^{p(\alpha/2+1)+\alpha/2},
\end{equation*}
which contradicts with (\ref{eq 9}).
\end{proof}    

\begin{proof}[\textbf{Proof of Corollary \ref{lm}}]
First it can be shown that any solution of (\ref{intro eq 2}) is also a solution of the integral equation
(\ref{equiv integral equation}), which is a direct result of Theorem 4.1 in \cite{li2016symmetry}. In fact the boundary condition of the partial differential equation in \cite{li2016symmetry} is too strong, and the boundary condition in (\ref{intro eq 2}) is sufficient.

% So $u$ is also a solution of the integral equation
% \begin{equation}
%     u(x)=\int_{R^n_+}G_{2m+\alpha}^+(x,y)u^p(y)dy.
% \end{equation}
Then by Theorem \ref{moving planes}, $u$ is trivial.
\end{proof}

\section{A priori estimates}\label{sec 3}

\begin{proof}[\textbf{Proof of Theorem \ref{main theorem}}]
Suppose $\|u\|_{L^\infty(\Omega)}\leq C$ does not hold. Then there exists a sequence of solutions $\{u_k\}$ to (\ref{in-1}), such that
\begin{equation}\label{contradiction}
	\max \limits _\Omega u_k\to \infty.
\end{equation}
Let
\begin{equation}
	H_{ki}(x):= (-\triangle)^iu_k(x),\,\, x\in \Omega ,\,\, i =0,1,\cdots m.
\end{equation}
	By the definition of $H_{km}$, it satisfies the equation
	 \begin{equation}
		\begin{cases}
			(-\triangle)^{\frac{\alpha}{2}}H_{km} = u_k ^p\geq 0, &in\,\, \Omega,\\
			H_{km}= 0, &on \,\, \Omega ^c.
		\end{cases}
	\end{equation}
	Combining with the {\itshape maximum principle} (see Theorem 1.1 in \cite{liu2022dirichlet}), $H_{km}\geq 0$ in $\Omega$. Similarly,
	$H_{k,m-1}$ satisfies the equation
	\[
	\begin{cases}
			(-\triangle)H_{k,m-1} = H_{km}\geq 0, &in\,\, \Omega,\\
			H_{k,m-1}= 0, &on \,\, \partial\Omega.
	\end{cases}
	\]
Therefore, $H_{k,m-1}\geq 0$ in $\Omega$ by the {\itshape maximum principle} for subharmonic functions. Repeating the process, we find out that, for each $i=0,1,\cdots ,  m$,
\begin{equation}
	H_{ki}\geq 0 \,\,\, in  \,\,\Omega .
\end{equation}

   Let\begin{equation}\label{para0}
   	\beta := \dfrac{p-1}{2m+\alpha},\,\,\, m_k :=\max \limits _{0\leq i\leq m}(\sup \limits _{\Omega} H_{ki})^{\frac{1}{1+2\beta i}},\,\,\, \lambda _k := m_k^{-\beta}.
   \end{equation}
   By (\ref{contradiction}), $m_k\to \infty$ as $k\to \infty$. For each $k$, $m_k$ is taken from $m+1$ values. Therefore, up to a subsequence, there exists an integer $0\leq j\leq m$, and $\{x^k\}\subset \Omega$, such that
   \begin{equation}\label{parameter}
   	m_k = (\sup \limits _{\Omega} H_{kj})^{\frac{1}{1+2\beta j}}:=
   	(H_{kj}(x^k))^{\frac{1}{1+2\beta j}}.
   \end{equation}

   Let \begin{equation}\label{vk}
   	v_k(x):=\dfrac{1}{m_k}u_k(\lambda _k x+x^k),
   \end{equation}
   then we have
   \begin{equation}
   	(-\triangle)^{\frac{\alpha}{2}+m}v_k(x) = v^p_k(x),\,\,\, x\in \Omega _k:=\left\{x\in\mathbb{R}^n \bigg | x= \dfrac{y-x^k}{\lambda _k}, y\in \Omega\right\}.
   \end{equation}

   Let $d_k = {\rm dist}(x^k,\partial \Omega)$. Notice that $\{d_k\}$ is bounded and $\lambda _k\to 0$ as $k\to \infty$. Up to a subsequence, the value of $\lim _{k\to\infty}\tfrac{d_k}{\lambda _k}$ has three possibilities ($\infty$, a positive constant, and 0). We will carry out the proof using the contradiction argument while exhausting all three possibilities.
   \\\\
   \textbf{Case (\romannumeral 1).} $\lim _{k\to\infty}\frac{d_k}{\lambda _k} = \infty$. Then
   \[
   \Omega _k\to \mathbb{R}^n\,\,\,\mbox{as}\,\, k \to \infty.
   \]
   We claim that there exists a function $v$ and a subsequence of $\{v_k\}$ (still denoted by $\{v_k\}$), such that as $k\to \infty$, for each $i=0,1,\cdots , m$,
   \begin{equation}\label{convergence}
   	(-\triangle)^{i}v_k(x)\to (-\triangle)^{i}v(x)\,\,\mbox{and}\,\,(-\triangle)^{\frac{\alpha}{2}+m}v_k(x)\to (-\triangle)^{\frac{\alpha}{2}+m}v(x)\,.
   \end{equation}
 thus\begin{equation}\label{R^n eq}
 	(-\triangle)^{\frac{\alpha}{2}+m}v(x)=v^p(x),\,\,\,x\in \mathbb{R}^n.
 \end{equation}

 In fact, if the claim holds, then (\ref{parameter}) and (\ref{vk}) imply that
   \[
     (-\triangle)^jv_k(0)=\dfrac{\lambda _k^{2j}}{m_k}H_{kj}(x^k)=\dfrac{H_{kj}(x^k)}{m_k^{1+2\beta j}}=1.
   \]
   Thus\begin{equation}\label{v(0)}
   	(-\triangle)^jv(0) = \lim\limits _{k\to\infty}(-\triangle)^jv_k(0) =1.
   \end{equation}
   However, a Liouville type theorem \cite{xu2022super} states that the only nonnegative solution to the equation (\ref{R^n eq}) is zero. This is a contradiction. Hence this case cannot happen.

   Now we begin to prove the claim. Conveniently, for each integer $k\geq 0$, we define
   \[
   C^{k,\gamma}=\begin{cases}
   	C^{k,\gamma} & \mbox {if}\,\,0<\gamma <1,\\
    C^{k+1,\gamma -1} &\mbox {if}\,\, 1<\gamma \leq 2.
   \end{cases}
   \]

   To verify (\ref{convergence}), we need to establish a uniform interior estimate for $v_k$ in a neighborhood of any point $x\in\mathbb{R}^n$, which is independent of both $k$ and $x$.

   From now on, we use $C$ to denote a universal constant whose values may vary from line to line.

   For any $x^0\in \mathbb{R}^n$, there exists an integer $N>0$ such that when $k>N$, we have $B_3(x^0)\subset \Omega _k$. Let
   \begin{equation}
   	\begin{cases}
   		h_{k1}(x)& : =(-\triangle)v_k(x)  = \dfrac{\lambda _k ^2}{m_k}H_{k1}(\lambda _kx+x^k) = \dfrac{H_{k1}(\lambda _k x + x^k)}{m_k^{1+2\beta}},\\
   		&\vdots\\
   		h_{ki}(x)& : =(-\triangle)h_{k,i-1}(x)=\dfrac{\lambda ^{2i}}{m_k}H_{ki}(\lambda _kx+x^k) = \dfrac{H_{ki}(\lambda _kx+x^k)}{m_k^{1+2\beta i}},\\
   		&\vdots\\
   	   g_k(x) & : = (-\triangle)^{\frac{\alpha}{2}}h_{km}(x) = (-\triangle)^{\frac{\alpha}{2}+m}v_k(x) = v_k^p(x),
   	\end{cases}
   	\quad\quad\mbox{in}\,\,\Omega _k.
   \end{equation}
   Equation (\ref{vk}) and (\ref{para0}) imply that
   \[
   	\| v_k \| _{L^\infty(\Omega _k)}\leq 1,
   \]
   thus
   \[
   	\| g_k \| _{L^\infty(\Omega _k)} = \| v_k  ^p\| _{L^\infty(\Omega _k)}\leq 1.
   \]
	Moreover, for any $x\in \Omega _k$ and $i=1, 2, \cdots , m$, 
	\[
	\| h_{ki}\| _{L^\infty(\Omega _k)} = \sup \limits _{x\in \Omega _k} \left | \dfrac{H_{ki}(\lambda _kx+x^k)}{m_k^{1+2\beta i}}\right |\leq 1.
	\]
	
	Let $\varphi$ be a smooth cutoff function such that $0\leq \varphi (x)\leq 1$ in $\mathbb{R}^n$, ${\rm {supp}}\,\varphi \subset B_3(x^0)$ and $\varphi (x)\equiv 1$ in $B_2(x^0)$. Define
	\begin{equation}
		\bar{h}_{km}(x):=c_{n,-\alpha}\int _{\mathbb{R}^n}\dfrac{\varphi (x)g_k(x)}{|x-y|^{n-\alpha}}dy.
	\end{equation}
	
	Then
	\[
	\bar{h}_{km} = (-\triangle)^{-\frac{\alpha}{2}}(\varphi g_k)=(-\triangle)^{1-\frac{\alpha}{2}}\circ (-\triangle)^{-1}(\varphi g_k).
	\]
	
	Let
	\[
	h_k(x):=(-\triangle)^{-1}(\varphi g_k)(x).
	\]
	Then
	\[
	h_k(x) = c_{n,-2}\int _{\mathbb{R}^n}\dfrac{\varphi (y)g_k(y)}{|x-y|^{n-2}}dy.
	\]
	By the $C^{1,\sigma}$ estimates for Poisson's equation, %加参考文献
	we derive that there exists a universal constant $C$ which is independent of $k$ and $x^0$ such that for any $0<\tau <1$,
	\begin{equation}
		\| h_k \| _{C^{1,\tau}({B_3(x^0)})}\leq C(\| \varphi g_k\| _{L^\infty(B_3(x^0))}+\| h_k \| _{L^\infty(B_3(x^0))} ).
	\end{equation}
	Applying Proposition 2.7 in \cite{silvestre}
	to the equation
    \[
    \bar{h}_{km} = (-\triangle)^{1-\frac{\alpha}{2}}h_k(x),
    \]
 we have that for all $0<\sigma <\min\{1,\tau +\alpha -1 \}$,
	\begin{equation}\label{barhkm norm}
		\| \bar{h}_{km} \| _{C^{0,\sigma}({B_3(x^0)})}<C. 
	\end{equation}
	
	Since
	\begin{equation}
		(-\triangle)^{\frac{\alpha}{2}}(h_{km}-\bar{h}_{km})=0, \quad x\in B_2(x^0),
	\end{equation}
	one can use Schauder interior estimates for the fractional Laplacian(more details in chapter 12 in \cite{chen2020fractional}) to derive that $h_{km}-\bar{h}_{km}$ is smooth in $B_{\frac{3}{2}}(x^0)$. Moreover, its $C^{0,\sigma}$ -norm can be  estimated from the $L^\infty$ -norm of $h_{km}-\bar{h}_{km}$, which is uniformly bounded. Combining with (\ref{barhkm norm}), we get	
	\begin{equation}\label{hkm norm}
		\| {h}_{km} \| _{C^{0,\sigma}({B_\frac{3}{2}(x^0)})}<C.
	\end{equation}

	Recall that \[\| {h}_{k,m-1} \| _{L^\infty(\Omega _k)}\leq 1,\]
	we use Schauder interior estimates in $B_{\frac{3}{2}}(x^0)$ to derive that
	\begin{equation}
	\begin{split}
		\| {h}_{k,m-1} \| _{C^{2,\sigma}}(B_{\frac{5}{4}}(x^0))& \leq C(\| {h}_{k,m-1}\| _{L^\infty({B_\frac{3}{2}(x^0)})}+\| {h}_{km} \| _{C^{0,\sigma}({B_\frac{3}{2}(x^0)})})\\
		&\leq C, \quad \mbox{here $C$ is independent of $k$ and $x^0$}.
	\end{split}
	\end{equation}
	  We can apply Schauder interior estimates in smaller balls and do the same process above to $h_{k,m-2},h_{k,m-3},\cdots $. Finally we find out that there exists a universal constant $C$ which is independent of $k$ and $x^0$ such that
	\begin{equation}\label{vk norm}
		\| v_k\| _{C^{2m,\sigma}(B_1(x^0))}\leq C.
	\end{equation}
	
	Due to the above uniform regularity estimate, it then follows from the {\itshape Arzel{\`a}-Ascoli theorem} that there exists a converging subsequence of $\{v_k\}$ in $B_1(x^0)$, denoted by $\{v_{1k}\}$. Moreover, for any muti index $\gamma$, $|\gamma|\leq 2m$,   $\{D^\gamma v_{1k}\}$ converges in $B_1(x^0)$. Then, similar to the process above, one can find a subsequence of $\{v_{1k}\}$, denoted by $\{v_{2k}\}$, such that for any muti index $\gamma$, $|\gamma|\leq 2m$,   $\{D^\gamma v_{2k}\}$ converges in $B_2(x^0)$. By induction, we get a chain of subsequence
	\[
	\{v_{1k}\}\supset \{v_{2k}\} \supset \{v_{3k}\} \cdots
	\]
	such that for any muti index $\gamma$, $|\gamma|\leq 2m$, $\{D^\gamma v_{ik}\}$ converges in $B_i(x^0)$ as $k\to\infty$. Now we take the diagonal sequence $\{v_{ii}\}$, then for any muti index $\gamma$, $|\gamma|\leq 2m$, $\{D^\gamma v_{ii}\}$ converges at all points in any $B_R(0)$. Thus we have constructed a subsequence of solutions(still denoted by $\{v_k\}$) that satisfies the first half of (\ref{convergence}).
	
	In order to prove the second half, recall that
	\[
	g_k:=(-\triangle)^{\frac{\alpha}{2}+m}v_k = v_k^p,
	\]
	from (\ref{vk norm}), we immediately derive, for any $x\in\mathbb{R}^n$ and for some $0<\sigma <1$,
	\[
	\| g_k\| _{C^{0,\sigma}(B_1(x))}\leq C,\quad \mbox{here $C$ is independent of $k$ and $x$}.
	\]
	By the Schauder's estimates on the equation
	\[
	h_k = (-\triangle)^{-1}(\varphi g_k),
	\]
	and similar to the argument in deriving (\ref{hkm norm}), we can show that there exists a constant $C$ independent of $k$ and $x$, such that
	\begin{equation}\label{hkm sigma+alpha norm}
		\| h_{km}\| _{C^{0,\sigma+\alpha}(B_1(x))}\leq C.
	\end{equation}
	
	By the definition of the fractional Laplacian and $h_{km}$, we have
	\begin{equation}
		\begin{split}
			(-\triangle)^{\frac{\alpha}{2}+m}v_k(x) & = (-\triangle)^{\frac{\alpha}{2}}h_{km}(x)\\
			&= \dfrac{C_{n,\alpha}}{2}\int _{\mathbb{R}^n}\dfrac{2h_{km}(x)-h_{km}(x+y)-h_{km}(x-y)}{|y|^{n+\alpha}}dy\\
			&= \dfrac{C_{n,\alpha}}{2}\bigg( \int _{\mathbb{R}^n\backslash B_1(0)}\dfrac{2h_{km}(x)-h_{km}(x+y)-h_{km}(x-y)}{|y|^{n+\alpha}}dy \\
			& + \int _{B_1(0)}\dfrac{2h_{km}(x)-h_{km}(x+y)-h_{km}(x-y)}{|y|^{n+\alpha}}dy \bigg)\\
			& = \dfrac{C_{n,\alpha}}{2}(I_1 + I_2).
		\end{split}
	\end{equation}
	For any $y\in B_1(0)$,
	\begin{equation}\label{I2}
		\dfrac{2h_{km}(x)-h_{km}(x+y)-h_{km}(x-y)}{|y|^{n+\alpha}}\leq\dfrac{C}{|y|^{n-\sigma}} \leqslant \dfrac{C}{|y|^{n-\sigma}} \cdot \dfrac{1}{1+|y|^{n+\alpha}},
	\end{equation}	
	here we use the equation (\ref{hkm sigma+alpha norm}) and notice that $\frac{1}{1+|y|^{n+\alpha}}\geq \frac{1}{2}$.
	
	For any $y\in \mathbb{R}^n\backslash B_1(0)$,
	\begin{equation}\label{I1}
		\dfrac{2h_{km}(x)-h_{km}(x+y)-h_{km}(x-y)}{|y|^{n+\alpha}}\leq \dfrac{C}{|y|^{n+\alpha}} = \dfrac{C}{1+|y|^{n+\alpha}}\cdot \dfrac{1+|y|^{n+\alpha}}{|y|^{n+\alpha}}\leq \dfrac{C}{1+|y|^{n+\alpha}}.
	\end{equation}
	
	Combining (\ref{I1}) and (\ref{I2}), we have that, for any $y\in\mathbb{R}^n$,
	\begin{equation}
		\begin{split}
			\dfrac{2h_{km}(x)-h_{km}(x+y)-h_{km}(x-y)}{|y|^{n+\alpha}} & \leq \dfrac{C}{1+|y|^{n+\alpha}}\left (1+\dfrac{1}{|y|^{n-\sigma}} \right)\\
			& := g(y).
		\end{split}
	\end{equation}
	Since $g\in L^1(\mathbb{R}^n)$, from the {\itshape Lebesgue's dominated convergence theorem} it yields that
	\[
	\begin{split}
		\lim\limits _{k\to\infty}(-\triangle)^{\frac{\alpha}{2}+m}v_k(x) & =
		\lim\limits _{k\to\infty} \dfrac{C_{n,\alpha}}{2}\int _{\mathbb{R}^n}\dfrac{2h_{km}(x)-h_{km}(x+y)-h_{km}(x-y)}{|y|^{n+\alpha}}dy\\
		& =  \dfrac{C_{n,\alpha}}{2}\int _{\mathbb{R}^n}\dfrac{2(-\triangle)^mv(x)-(-\triangle)^mv(x+y)-(-\triangle)^mv(x-y)}{|y|^{n+\alpha}}dy\\
		& = (-\triangle)^{\frac{\alpha}{2}+m}v(x).
	\end{split}
	\]
	This proves (\ref{convergence}) and (\ref{R^n eq}).
	\\\\
	\textbf{Case (\romannumeral 2).} $\lim _{k\to\infty}\frac{d_k}{\lambda _k} = C >0$. Then
	\[
   \Omega _k\to \mathbb{R}^n_{+C}:=\{x_n\geq -C \,|\, x\in\mathbb{R}^n \}\,\,\,\mbox{as}\,\, k \to \infty.
   \]
   Similar to Case (\romannumeral 1), here we are able to establish the existence of a function $v$ and a subsequence of $\{v_k\}$, such that as $k\to \infty$, for each $i=0,1,\cdots , m$,
   \begin{equation}\label{convergence2}
   	(-\triangle)^{i}v_k(x)\to (-\triangle)^{i}v(x)\,\,\mbox{and}\,\,(-\triangle)^{\frac{\alpha}{2}+m}v_k(x)\to (-\triangle)^{\frac{\alpha}{2}+m}v(x)\,,
   \end{equation}
   and thus
   \begin{equation}\label{halfspace eq}
   	(-\triangle)^{\frac{\alpha}{2}+m}v(x) = v^p(x),\,\,\,x\in \mathbb{R}^n_{+C}.
   \end{equation}

   According to Corollary \ref{lm},
   (\ref{halfspace eq}) has no positive solution. This contradicts with (\ref{v(0)}). Next we prove (\ref{convergence2}) and (\ref{halfspace eq}).

   Let $D_1 = B_1(0)\cap \{x_n \geq 0\}$. Then similar to the argument as in Case (\romannumeral 1), one can show that there exists a subsequence $\{v_{1k}\}$ of $\{v_k\}$, such that for each $i = 0,1,\cdots , m$ and for any $x\in D_1$,
   \[
   	(-\triangle)^{i}v_{1k}(x)\to (-\triangle)^{i}v(x)\,\,\mbox{and}\,\,(-\triangle)^{\frac{\alpha}{2}+m}v_{1k}(x)\to (-\triangle)^{\frac{\alpha}{2}+m}v(x)\,.
   \]

   Let $D_2 = B_2(0)\cap \{x_n \geq -\frac{C}{2}\}$. We can find a subsequence $\{v_{2k}\}$ of $\{v_{1k}\}$, such that for each $i = 0,1,\cdots , m$ and for any $x\in D_2$,
   \[
   	(-\triangle)^{i}v_{2k}(x)\to (-\triangle)^{i}v(x)\,\,\mbox{and}\,\,(-\triangle)^{\frac{\alpha}{2}+m}v_{2k}(x)\to (-\triangle)^{\frac{\alpha}{2}+m}v(x)\,.
   \]
   Repeating the above process and take the diagonal sequence $\{v_{kk}\}$. It is easy to see that for each $i = 0,1,\cdots , m$ and for any $x\in \mathbb{R}^n_{+C}$,
   \[
   	(-\triangle)^{i}v_{kk}(x)\to (-\triangle)^{i}v(x)\,\,\mbox{and}\,\,(-\triangle)^{\frac{\alpha}{2}+m}v_{kk}(x)\to (-\triangle)^{\frac{\alpha}{2}+m}v(x)\,.
   \]
   This verifies (\ref{convergence2}) and (\ref{halfspace eq}).
   \\\\
	\textbf{Case (\romannumeral 3).} $\lim _{k\to\infty}\frac{d_k}{\lambda _k} = 0$.
	
	In this case, there exists a point $x^0\in\partial \Omega$ and a subsequence $\{x^k\}$, still denoted by $\{x^k\}$, such that
	\[
	x^k\to x^0\,\,\,\mbox{as}\,\,\,k\to\infty.
	\]
	Let $p^k = \frac{x^0-x^k}{\lambda _k}$, $h_{k0}= v_k$.  Obviously,
	\begin{equation}
		h_{ki}(p^k)=0,\,\,\,i=0,1,2,\cdots , m,
	\end{equation}
    and
    \[
    p^k\to 0\,\,\,\mbox{as}\,\,\,k\to\infty.
    \]

    Next we will show that $h_{kj}$ is uniformly H{\"o}lder continuous near $p^k$, i.e.
    \begin{equation}\label{holder}
    	|h_{kj}(x)-h_{kj}(p^k)|\leq C|x-p^k|^{\gamma}.
    \end{equation}
    where
    \begin{equation}
    	\gamma = \begin{cases}
    		1, & \mbox{if $0\leq j<m$,}\\
    		\frac{\alpha}{2},  & \mbox{if $j=m$.}
    	\end{cases}
    \end{equation}
     We postpone the proof of (\ref{holder}) for a moment. Note that
   \begin{equation}\label{df=1}
   	h_{kj}(0)-h_{kj}(p^k) =1.
   \end{equation}
   On the other hand, it follows from (\ref{holder}) that
   \begin{equation}
   	|h_{kj}(0)-h_{kj}(p^k)|\to 0,\,\,\,\mbox{as}\,\,\, k\to \infty.
   \end{equation}
   This contradicts (\ref{df=1}), and thus rules out the possibility of Case (\romannumeral 3).

   To prove (\ref{holder}), we construct an auxiliary function $\varphi$ such that for any $x$ near $p^k$, it holds that
   \begin{equation}
   	|h_{kj}(x)-h_{kj}(p^k)|\leq |\varphi (x) - \varphi (p^k)|\leq C|x-p^k|^{\gamma}.
   \end{equation}

   Since $\partial \Omega$ is $C^2$, for $k$ sufficiently large, there exists a unit ball contained in $\Omega _k^c$ that is tangent to $\partial \Omega _k$ at $p^k$. Without loss of generality, we may assume that the ball is centered at 0.

   If $0\leq j <m$, we have that
   \[
   \begin{split}
   	&(-\triangle)h_{kj} = h_{k,j+1},\quad \mbox{in $\Omega _k$,}\\
   	&|h_{k,j+1}(x)|\leq 1.
   \end{split}
   \]
   Let
   \begin{equation}
   	\psi _1(x): = C(1-|x|^2)_+ =\begin{cases}
   		C(1-|x|^2), & x\in B_1(0)\\
   		0, & x\in B_1^c(0).
   	\end{cases}
   \end{equation}	
	choose constant $C$ such that
   	\begin{equation}
   		(-\triangle)\psi _1(x)=1.
   	\end{equation}
   	
   	Let
   	\begin{equation}
   		\psi _2(x) = \dfrac{1}{|x|^{n-2}}\psi _1\left(\dfrac{x}{|x|^2} \right)
   	\end{equation}
   	be the Kelvin transform of $\psi _1(x)$. Then
   	\[
   	(-\triangle)^{\frac{\alpha}{2}}\psi _2(x) = \dfrac{1}{|x|^{n+2}},\,\,\, x\in B_1^c(0).
   	\]
   	Let $\xi (x)$ be a smooth cutoff function such that $0\leq \xi (x)\leq 1$ in $\mathbb{R}^n$, $\xi (x) = 0$ in $B_1(0)$, $\xi (x)=1$ in $B_3(0)^c$, and
   	 \[
   	(-\triangle)\xi (x) \geq -C.
   	\]
   	Let
   	\begin{equation}
   		\varphi (x) = t\psi _2(x)+\xi (x), \quad\mbox{$t$ is a positive constant},
   	\end{equation}
   	and denote $D=(B_3(0)\backslash B_1(0))\cap \Omega _k$. For $t$ sufficiently large and $x\in D$,
   	\begin{equation}
   		\begin{split}
   			(-\triangle) \varphi (x) & = t(-\triangle)\psi _2(x) + (-\triangle)\xi(x)\\
   			& \geq \dfrac{t}{|x|^{n+2}}-C\\
   			& \geq 1\geq h_{k,j+1}.
   		\end{split}
   	\end{equation}
   	
   	Next, we prove that $\varphi \geq h_{kj}$ on $\partial D$. Notice that
   	\[
   	\partial D = (\partial D \cap \partial \Omega _k)\cup (\partial D\cap \partial B_3(0)).
   	\]
   	For $x\in (\partial D \cap \partial \Omega _k)$, the conclusion is plain since $h_{kj}(x) = 0$.
   	For any $x\in (\partial D\cap \partial B_3(0))$,
   	\[
   	\varphi (x)\geq \xi (x)\geq 1 \geq h_{kj}(x).
   	\]
   	Using the maximum principle, we have that
   	\[
   	\varphi\geq h_{kj},\quad\mbox{in $D$}.
   	\]
   	
   To show that $\varphi$ is H\"older continuous near $p^k$, it suffices to show that $\psi _2(x)$ is H\"older continuous in $D$. Indeed,
   \[
   \psi _2(x)-\psi _2(p^k)=|\psi _x(x)|\leq C(|x|-1)\leq C|x-p^k|.
   \]
   Thus, there exits a constant $C>0$, such that
   \[
   \varphi(x)-\varphi (p^k)\leq C|x-p^k|.
   \]
   Recall that $h_{kj}(p^k)=0$, then
   \begin{equation}
   	  	0\leq h_{kj}(x)-h_{kj}(p^k) = h_{kj}(x)\\
   	  	\leq \varphi (x)  = \varphi (x)-\varphi(p^k)\leq C|x-p^k|.
   \end{equation}

   If $j=m$, one can show the similar result by setting
   \begin{equation}
   	\psi _1(x) = C(1-|x|^2)^{\frac{\alpha}{2}}_+
   \end{equation}
   and
   \begin{equation}
   	\psi _2(x) = \dfrac{1}{|x|^{n-\alpha}}\psi _1\left(\dfrac{x}{|x|^2} \right).
   \end{equation}
   This proves (\ref{holder}) and completes the proof of Case (\romannumeral 3).
\end{proof}

\begin{proof}[\textbf{Proof of Corollary \ref{corollary}}]
Suppose the conclusion does not hold. Then there exists a sequence of solutions $\{w_k\}$ to (\ref{in-1}), such that
\begin{equation}
	\max \limits _\Omega (-\Delta)^iw_k\to \infty. 
\end{equation}
We can replace $u_k$ with $w_k$ in proof of theorem \ref{main theorem}, and the rest of the proof is totally same.
\end{proof}

\bibliographystyle{siam}
\bibliography{ref}
\end{document}